\documentclass[10pt,a4paper]{article}

\usepackage[T1]{fontenc}
\usepackage[utf8]{inputenc}
\usepackage{lmodern}
\usepackage{amsthm}
\usepackage{amsmath}
\usepackage{amsfonts}
\usepackage{amssymb}
\usepackage{graphicx}

\newtheorem{theorem}{Theorem}[section]
\newtheorem{lemma}[theorem]{Lemma}
\newtheorem{proposition}[theorem]{Proposition}
\newtheorem{corollary}[theorem]{Corollary}
\theoremstyle{definition}
\newtheorem{remark}{Remark}[section]
\newtheorem{definition}{Definition}

\usepackage{geometry}
\title{Action of $\mathrm{M}(0,2n)$ on some kernel spaces coming from $\mathrm{SU}(2)$-TQFT}
\author{Ramanujan Santharoubane}
\date{}

\begin{document}

\maketitle
\begin{abstract}  For $n$ an even number, we study representations of the mapping class group of the $n$-punctured sphere arising from $\mathrm{SU}(2)$-TQFT when all punctures are colored by the same integer $N \geq 1$. We prove that the conjecture of Andersen, Masbaum and Ueno (stated in \cite{AMU}) holds for the $4$-punctured sphere for all $N \geq 2$. In the case $n \geq 6$ of punctures, we prove it for the pseudo-Anosovs satisfying a homological condition, namely they should act with a non trivial stretching factor on certain eigenspaces of homology of a $\frac{n}{2}$-fold branched cover considered by McMullen. The main idea is to consider the \emph{kernel space} which is the kernel of the natural map from the skein module to the $\mathrm{SU}(2)$-TQFT. Our main theorem identifies, as representations of mapping class groups, certain of these kernel spaces with homology eigenspaces considered by McMullen. Our results concerning the AMU conjecture are obtained by taking appropriate limits of the quantum representations in which such a kernel space may appear as a subspace of the limit representation. 
 \end{abstract}
\section{Introduction} 
Throughout this paper $n \geq 4$ is an even integer. We consider quantum representations of the mapping class group of the $n$-punctured sphere arising from Witten-Reshetikhin-Turaev $\mathrm{SU}(2)$-Topological Quantum Field Theory (TQFT). The motivation of this work is the AMU conjecture stated by Andersen, Masbaum and Ueno (see \cite{AMU}). This conjecture relates quantum representations to the Nielsen-Thurston classification of surface homeomorphisms. This is part of a more general question which is to know what kind of geometric information can be extracted from quantum representations. At the moment, results concerning the AMU conjecture are only known in genus zero and one (see \cite{AMU}, \cite{JensSoren} and \cite{Santharoubane}). In this paper we give a new method to study the AMU conjecture and we use it to prove the conjecture in some new cases.

For $N \geq 1$ , $r-1 \geq N \geq 1$ and $A_r$ a $4r$-th primitive root of unity we denote by $S_{A_r}(B^3,(N)_n)$ the skein module, evaluated at $A_r$, of the $3$-ball whose boundary is equipped with $n$ banded points all colored by $N$. The mapping class group of the $n$-punctured sphere (denoted by $\mathrm{M}(0,n)$) acts projectively on $S_{A_r}(B^3,(N)_n)$.

\noindent The quantum representation is defined on the space $\mathrm{V}_{2r}(S^2,(N)_n)$ (we use the notations of \cite{BHMV}, it corresponds to the $\mathrm{SU}(2)$-TQFT at level $k=r-2$ in the construction of TQFT via geometric quantization or Conformal Field Theory). In the construction of Blanchet, Habegger, Masbaum and Vogel, the action of $\mathrm{M}(0,n)$ on the space $S_{A_r}(B^3,(N)_n)$ is the starting point of the construction of the quantum representation of $\mathrm{M}(0,n)$ on $\mathrm{V}_{2r}(S^2,(N)_n)$. To describe this space we recall the following definition
\begin{definition} Let $\mathrm{K}_r(B^3,(N)_n)$ be the left kernel of the natural sesquilinear form on the space $S_{A_r}(B^3,(N)_n)$ defined by the Kauffman bracket.
 \end{definition}
\noindent  It follows from \cite[Prop 1.9 ]{BHMV} that the space $\mathrm{K}_r(B^3,(N)_n)$ is stable for  the action of $\mathrm{M}(0,n)$. And the quantum representation is defined by $$\mathrm{V}_{2r}(S^2,(N)_n) \, = \,   S_{A_r}(B^3,(N)_n) /  \mathrm{K}_r(B^3,(N)_n)$$
 
\noindent  For $2r \geq  Nn+2$ the space $\mathrm{K}_r(B^3,(N)_n)$ is simply zero. But for $2r < Nn+2$ this space is not trivial and is actually the part in the skein module which is  usually neglected in the TQFT. Throughout this paper we will call this space the \emph{kernel space}. Our major concern is to see what kind of asymptotic TQFT-like information we can extract from these kernel spaces. Surprisingly, these objects which come from quantum topology are connected to geometric representations studied for a long time. More precisely the projective representation induced on $\mathrm{K}_r(B^3,(N)_n)$ in the case $2r=nN$ is exactly the same as the one induced on a space $\mathrm{H}^1(X)_q$ (depending on a root of unity $q$) considered by McMullen in \cite{mc} which is a subspace of the first cohomology space of a surface $X$ which is a cyclic branched covering of the sphere. Here is our main theorem :
\begin{theorem} \label{main_theorem_punctured_spheres} When $2r = Nn$ and $A_r$ is a $4r$-th primitive root of unity, if we set $q = A_r^{4N}$, then  $$ \mathrm{K}_r(B^3,(N)_n) \, \simeq \,   \mathrm{H}^1(X)_{q^{-1}}$$ as projective representations of $\mathrm{M}(0,n)$.
\end{theorem}

At this point we need to make some remarks. First the representations of $\mathrm{M}(0,n)$ on $\mathrm{H}^1(X)_q$ are closely related to the Burau representations (see \cite[Theorem 5.5]{mc} and Remark \ref{remark mc}) and the surface $X$ is always a cyclic branched $\frac{n}{2}$-covering of the sphere in Theorem \ref{main_theorem_punctured_spheres}.

  In the case $N=1$ and $A^4=-1$ the action of $\mathrm{M}(0,n)$ on the skein module $S_A(B^3,(1)_n)$ has already been studied (see \cite{AMU}, \cite{JensSoren} and \cite{Kasahara}). Loosely speaking the representations were interpreted using an algebraic operation in the homology of a cyclic branched $2$-covering of the sphere. This homological space is nothing but $\mathrm{H}^1(X)_{-1}$ in the setting of \cite{mc}. 

The case $N =1$ and $A^4 = -1$ is actually an extreme case where the kernel space is so big that it is the entire skein module. In our setting, this only happens when $n=4$ and $N=1$, where the cyclic  branched covering of the sphere we consider is simply the torus and the cover is given  by quotienting by a hyperelliptic involution. Therefore, if we set $N=1$ and $n=4$ in Theorem \ref{main_theorem_punctured_spheres}, we recover the result shown in \cite{AMU}.

For $(n,N) \neq (4,1)$, Theorem \ref{main_theorem_punctured_spheres} generalises the results of \cite{AMU}, \cite{JensSoren} and \cite{Kasahara} which only concern the case $N=1$.

  For the proof of Theorem \ref{main_theorem_punctured_spheres}, we have to distinguish (see Remark \ref{r1}) the case $n=4$ (see Section \ref{proof_n_4}) and the case $n \geq 6$ (see Section \ref{proof_n_6}). The proof for $n=4$ relies on a recursive formula for Jones-Wenzl idempotents proved by Frenkel and Khovanov (see \cite[Theor~3.5]{FK}) and independently proved by Morrison (see \cite{Morrison}).

Finally we recall that Lawrence (see \cite{Lawrence1}, \cite{Lawrence2}) gave an cohomological interpretation of the quantum representation of the braid group $B_n$. Briefly speaking, it involves the action of $B_n$ on a certain cohomological space (with local coefficients) associated to the configuration space of $n$ points on the $2$-disc. In the same direction, we also refer to the work done by Bigelow (see \cite{Bigelow}). We don't know how to prove Theorem \ref{main_theorem_punctured_spheres} using these cohomological approaches. It would be interesting to understand Theorem \ref{main_theorem_punctured_spheres} in these settings.
\\

 As an application of our main theorem we give new cases where the AMU conjecture can be proved. In the case $n=4$ we prove that the AMU conjecture holds for $N \geq 2$. The case $N=1$ was already proved in $\cite{AMU}$. The precise statement is as follows :
\begin{corollary} \label{cor1_c3}
For all $N \geq 1$,  if $\phi \in \mathrm{M}(0,4)$ is  pseudo-Anosov then there exists $r_0$ (depending on $N$ and $\phi$) such that for all $r \geq r_0$, $\phi$  has infinite order when acting on $\mathrm{V}_{2r}(S^2,(N)_4) )$
\end{corollary}

For the case $n \geq 6$, Theorem \ref{main_theorem_punctured_spheres} is not enough to prove the AMU conjecture for all pseudo-Anosov elements. However, we prove that the AMU conjecture holds for pseudo-Anosov with non trivial stretching factors when acting on the cohomological spaces $\mathrm{H}^1(X)_q$. More precisely, if we denote by $\rho_q : \mathrm{M}(0,n)  \to \mathrm{PAut}(\mathrm{H}^1(X)_q )$ the projective representation defined in \cite{mc}, for any element $\phi \in \mathrm{M}(0,n)$, the operator $\rho_q ( \phi) \in \text{Aut}(\mathrm{H}^1(X)_q )$ is defined up to a phase factor which is a root of unity. Therefore it makes sense to say that $\rho_q ( \phi)$ has a spectral radius strictly greater than one. We prove the following

\begin{corollary} \label{cor1.2_c3} Let $N \geq 1$ and let $n \geq 4$ be even. Let $\phi \in \mathrm{M}(0,n)$ be pseudo-Anosov. Suppose that there exists $q$ a $\frac{n}{2}$-th primitive root of unity  such that the operator $\rho_{q}(\phi)$ acting on the space $\mathrm{H}^1(X)_{q} $ has a spectral radius strictly greater than one. Then there exists $r_0$ (depending on $N$, $n$ and $\phi$) such that for all $r \geq r_0$, $\phi $ has infinite order when acting on $\mathrm{V}_{2r}(S^2,(N)_n )$.

\end{corollary}

 Corollary \ref{cor1_c3} and \ref{cor1.2_c3} are deduced directly from Theorem \ref{main_theorem_punctured_spheres} using a limit argument. 
The strategy is to notice that the action of $\mathrm{M}(0,n)$ on $\mathrm{V}_{2r}(S^2,(N)_n)$ admits a limit as $r \to \infty$ and as $A_r$ tends to a specific root of unity. In general the limit representation is not arising from TQFT. The action of $\mathrm{M}(0,n)$ on a kernel space may appear as subrepresentation of this limit. Moreover, a geometric interpretation of such a kernel space gives results concerning the AMU conjecture. With hindsight, the proofs of known cases of the AMU conjecture in \cite{AMU} and \cite{JensSoren} can also be phrased in terms of kernel spaces. The difference is that in \cite{AMU} and \cite{JensSoren}, the entire limit representation is a kernel space in our sense, while in the present paper only a subspace of the limit representation is a kernel space.
\\

The paper is organized as follows. In Section 2 we prove that $\dim(\mathrm{K}_{Nn/2}(B^3,(N)_{n})) = n-2$ (see Theorem \ref{cor3_c3}) and in Subsection \ref{action_MCG} we recall the definition of the action of $\mathrm{M}(0,n)$ on a certain basis of $S_{A_r}(B^3,(N)_n)$.  In Section 3, we briefly recall the definition of the representations considered by McMullen in \cite{mc} and we derive explicit formulas (see Section \ref{Explicit formulas}) which we summarize in Conclusion 1. Section 4 is devoted to the proof of Theorem \ref{main_theorem_punctured_spheres} in the case $n \geq 6$. This proof is done by giving explicit formulas for the actions of the canonical generators of $\mathrm{M}(0,n)$ on $\mathrm{K}_{Nn/2}(B^3,(N)_{n})$ and comparing it with the formulas of Conclusion 1. Section 5 deals with the case $n=4$. This case has to be handled separately (see Remark \ref{r1}) and the basis of $\mathrm{K}_{2N}(B^3,(N)_{4})$ we construct uses a recursive formula for Jones-Wenzl idempotents proved by Frenkel and Khovanov in \cite{FK} and independently by Morrison in \cite{Morrison}. Section 5 ends with the proof of Corollary \ref{cor1_c3} and Corollary  \ref{cor1.2_c3}. Finally in Section 6, we give examples (using experimental computations) where Corollary \ref{cor1.2_c3} may be applied. We also compare the homological criterion of Corollary \ref{cor1.2_c3} with the homological criterion found by Egsgaard and Jorgensen in \cite{JensSoren} and show that the two are independent.

\paragraph{Acknowledgements.}I would like to thank S. Bigelow, C. Blanchet, F. Costantino, R. Lawrence, J. Marché, G. Masbaum for helpful discussions.

\section{TQFT and the kernel space}
For this section let $N \geq 1$ and $n \geq 4$ even.
\subsection{Review of the TQFT for the torus}
In this subsection we recall elementary facts about the $\mathrm{SU}(2)$-TQFT for the torus. In the interest of brevity, we have kept the details to a minimum and we refer to \cite{BHMV} for more details. We also refer to \cite{Kauffman} and \cite{MV} for more details on skein calculus. Facts given in this subsection are crucial to find the dimensions of kernel spaces (see Subsection \ref{dim}).

Let $S(D^2 \times S^1)$ be the skein module of the solid torus over the ring of Laurent polynomials $\mathbb{C}[A,A^{-1}]$, it has a natural multiplicative structure. Moreover as a $\mathbb{C}[A,A^{-1}]$-algebra $$ S(D^2 \times S^1) \simeq \, \mathbb{C}[A,A^{-1}][z]$$ where $\mathbb{C}[A,A^{-1}][z]$ is the $\mathbb{C}[A,A^{-1}]$-algebra of polynomials in one variable $z$. Recall that $z$ is the following banded knot in $D^2 \times S^1$
$$z \quad =\begin{minipage}[c]{2cm}
\includegraphics[scale = 0.1]{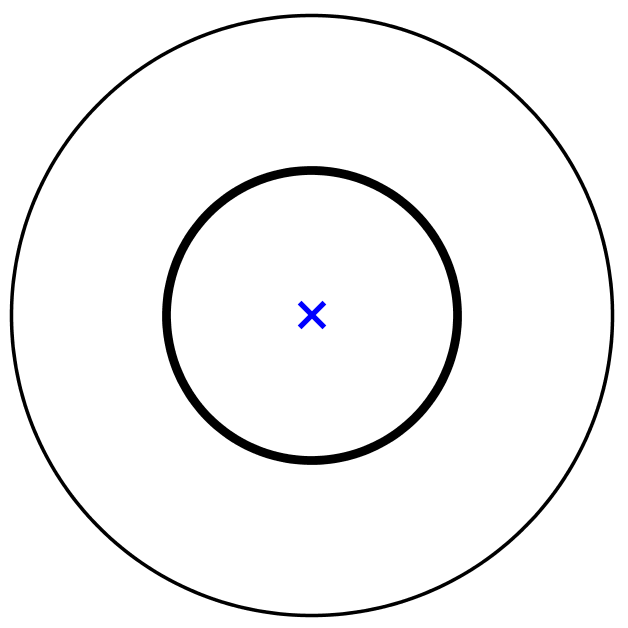}
\end{minipage}$$
and $z^m$ is $m$-parallel copies of $z$. For $l \geq 0$, we denote by $e_l$ the $l$-th Chebyshev polynomial determined by $$ e_0 = 1, \quad  e_1 = z \quad \text{and}  \quad  ze_l = e_{l+1} + e_{l-1} \quad \text{for} \quad  l \geq 2$$ For later use, we keep in mind that $ e_l(x+x^{-1}) = \dfrac{x^{l+1}-x^{-l-1}}{x-x^{-1}} $ for all $x \in \mathbb{C}-\{0 \}$ and  $  l \geq 0$ .
Let $r \geq 2$, it is known that $$ \mathrm{V}_{2r}(\mathbb{T}^2) \simeq \, \mathbb{C}[z] / (e_{r-1}) $$ The image of $z$ in 
$ \mathrm{V}_{2r}(\mathbb{T}^2)$ will also be denoted by $z$. For $0 \leq l \leq r-2$ we denote by $z_l$ the banded knot $z$ colored by $l$. Remark that $$ Z_r(D^2 \times S^1 , z_l) = e_l$$ where $Z_r$ is the Witten-Reshetikhin-Turaev TQFT functor. The natural hermitian form on $\mathrm{V}_{2r}(\mathbb{T}^2)$ is denoted by $ \langle \, , \rangle_{r, \mathbb{T}^2}$ and with respect to this form, $e_0 ,...., e_{r-2}$ is an orthonormal basis. We recall the following lemma proved in \cite[Lemma~6.3]{BHMV0} 
\begin{lemma} \label{l1_c3}
In $\mathrm{V}_{2r}(\mathbb{T}^2)$ we have the following periodicity 
$$\forall \,  b \geq 0,  \quad  e_{2r+b} = e_b \quad  \quad \text{and} \quad \quad \forall a \in \{ 0,..., r-1 \}, \quad e_{r+a} = - e_{r-2-a} $$
 with the convention that $e_{-1} =0$.

\end{lemma}

\subsection{Dimension of kernel spaces for punctured spheres} \label{dim}
For $k \geq 0$ we denote by $(S^2,(N)_n,(k))$ the $2$-sphere equipped with $n$ marked points with color $N$ and one marked point with color $k$. When $k=0$, we simply write $(S^2,(N)_n)$ for $(S^2,(N)_n,(0))$. For $k \geq 0$ let $S(B^3,(N)_n,(k))$ be the skein module of the $3$-ball whose boundary is $(S^2,(N)_n,(k))$, and we also write  $S(B^3,(N)_n)$ for  $S(B^3,(N)_n,(0))$. For $t$ a non zero complex number, we denote by $S_t(B^3,(N)_n,(k))$ the skein module evaluated at $A = t$. It is a finite dimensional complex vector space. Since $S(B^3,(N)_n,(k))$ is a free module over $\mathbb{C}[A,A^{-1}]$ we have  $$  \dim_{\mathbb{C}[A,A^{-1}]} S (B^3,(N)_n,(k)) =\dim_{\mathbb{C}} S_t(B^3,(N)_n,(k)) \quad \text{for all} \quad t \in \mathbb{C}^*$$ Let $r \geq 2$ and $A_r$ be a $4r$-th primitive root of the unity. If $r-1 \geq N,k  \geq 0$, the Kauffman bracket gives a natural sesquilinear form on  $S_{A_r}(B^3,(N)_n,(k))$ which we denote by $\langle \, , \rangle_{A_r}$. This form plays a key role in the construction of TQFT. Indeed, from \cite{BHMV}, we have $$ \mathrm{V}_{2r}(S^2,(N)_n,(k)) \, =  S_{A_r}(B^3,(N)_n,(k)) /  \mathrm{K}_r(B^3,(N)_n,(k))$$ where $ \mathrm{K}_r(B^3,(N)_n,(k))$ is the left kernel of the form $\langle \, , \rangle_{A_r}$ and $\mathrm{V}_{2r}(S^2,(N)_n,(k)) $ is the space associated to $(S^2,(N)_n,(k))$ by the $\mathrm{SU}(2)$-TQFT at $A_r$.
\\

  It can been seen that for $2r \geq Nn+2 $, $\dim( \mathrm{K}_r(B^3,(N)_n)) =0$. The goal of this subsection is to compute this dimension for $2r = Nn$. Here is the main theorem of this subsection 

\begin{theorem} \label{cor3_c3}
If $2r = Nn$ we have

\begin{center}
    $ \dim(\mathrm{K}_r(B^3,(N)_n))= n-2$  \quad
\text{and} \quad
  $ \dim(\mathrm{K}_r(B^3,(N)_{n-2},(2N-2))) = 1 $
\end{center}

\end{theorem}

\begin{remark}
 As we will see in Proposition \ref{l2_c3}, the second statement of Theorem \ref{cor3_c3} will play an important role when computing the action of $\mathrm{M}(0,n)$ on $\mathrm{K}_r(B^3,(N)_n)$.
\end{remark}
The proof of Theorem \ref{cor3_c3} needs several steps. First we have to understand what the dimension of $\mathrm{V}_{2r}(S^2,(N)_n,(k))$ means in terms of pairing in $\mathrm{V}_{2r} (\mathbb{T}^2)$.

\begin{proposition} \label{prop1_c3} Let $r \geq 2$ and $r-1 \geq N,k  \geq 0$. Then $$\dim( \mathrm{V}_{2r}(S^2,(N)_n,(k))) = \langle e_k , (e_N)^n \rangle_{r, \mathbb{T}^2} $$
\end{proposition}

\begin{proof} Let $P = \{ p_1,...,p_n,p_{n+1} \}$ be a set of $n+1$ banded points on $S^2$. Let $(S^2 \times I, P_{N,k})$ be the cobordism $S^2 \times I$ equipped with the banded arcs $(p_i \times I)_{i = 1,...,n}$ colored by $N$ and the arc $p_{n+1} \times I$ colored by $k$. We can then define the operator $Z_r (S^2 \times I, P_{N,k}) \in \text{End}( \mathrm{V}_{2r}(S^2,(N)_n,(k)))$ which is nothing but the identity operator, therefore $$ \dim( \mathrm{V}_{2r}(S^2,(N)_n,(k))) = \text{tr}( Z_r (S^2 \times I, P_{N,k})) = Z_r( S^2 \times S^1 , \tilde{P}_{N,k} )$$ where $( S^2 \times S^1 , \tilde{P}_{N,k})$ is the $3$-manifold without boundary $S^2 \times S^1$ equipped with the banded link $(p_i \times S^1)_{i = 1,...,n}$ colored by $N$ and the banded knot $p_{n+1} \times S^1$ colored by $k$. 

Moreover recall that $S^2 \times S^1$ can be obtained by gluing two solid tori along their boundary with the identity map, so $$
Z_r( S^2 \times S^1 , \tilde{P}_{N,k} )  = \langle  Z_r(D^2 \times S^1 , z_k) ,Z_r(D^2 \times S^1 , z_N^n ) \rangle_{r,\mathbb{T}^2}  = \langle  e_k , (e_N)^n \rangle_{r, \mathbb{T}^2}$$
\end{proof}
\begin{corollary}\label{cor2_c3} If we write $(e_N(X))^n = \sum_{k = 0}^{nN} c(k,N,n) e_k(X)$ in $\mathbb{Z}[X]$ then $$c(k,N,n) = \dim( S(B^3,(N)_n, (k)))$$
\end{corollary}
\begin{proof}
Let  $r \geq Nn+2$. The dimension of $ S(B^3,(N)_n, (k))$ is the same as  $\dim( V_{2r}(S^2,(N)_n,(k)))$ which equals $ \langle e_k(z) , (e_N(z))^n \rangle_{r, \mathbb{T}^2}$ by Proposition \ref{prop1_c3}. We deduce the following $$ \dim( S(B^3,(N)_n, (k))) =\sum_{l = 0}^{nN} c(l,N,n)  \langle e_l(z) , e_k(z) \rangle_{r, \mathbb{T}^2}$$ We conclude since $   r-1> Nn$ and for all $Nn \geq k,l \geq 0$ : $ \langle e_l(z) , e_k(z) \rangle_{r, \mathbb{T}^2}=1$ when $k = l$ and 0 otherwise. 
\end{proof}

\begin{lemma} \label{1.2_c3}One has $ c(Nn-2,N,n) =  n-1$ and  $c(Nn,N,n) =1 $.

\end{lemma}
\begin{proof} This is shown by elementary skein calculations left to the reader. \end{proof}

\begin{proof}[\textbf{Proof of Theorem \ref{cor3_c3}} ] We set $2r = Nn$. Let us compute $ \dim(\mathrm{V}_{2r}(S^2,(N)_n))$ using Proposition \ref{prop1_c3} $$ \dim(\mathrm{V}_{2r}(S^2,(N)_n))  =  \langle e_N^n ,e_0 \rangle_{r,\mathbb{T}^2}  = \sum_{k = 0}^{Nn} c(k,N,n) \langle e_k ,e_0 \rangle_{r,\mathbb{T}^2} = \sum_{k = 0}^{2r} c(k,N,n) \langle e_k ,e_0 \rangle_{r,\mathbb{T}^2}$$ By Lemma \ref{l1_c3} and using that $e_0,..., e_{r-2}$ is an orthonormal basis of $\mathrm{V}_{2r}(\mathbb{T}^2)$ we notice that the only terms remaining are for $k=0,2r-2$ and $2r$ :
\begin{align}
\dim(\mathrm{V}_{2r}(S^2,(N)_n)) &=  c(0,N,n)  \langle e_0 ,e_0 \rangle_{r,\mathbb{T}^2} +c(2r-2,N,n) \langle e_{2r-2} ,e_0 \rangle_{r,\mathbb{T}^2}+c(2r,N,n) \langle e_{2r} ,e_0 \rangle_{r,\mathbb{T}^2} \notag \\
& =  c(0,N,n) -c(Nn-2,N,n) +c(Nn,N,n) \notag 
\end{align}
 Now using Corollary \ref{cor2_c3} and Lemma \ref{1.2_c3} to replace the three terms of this sum, we get
$$ \dim(\mathrm{V}_{2r}(S^2,(N)_n)) = \dim(S(B^3,(N)_n)) - (n-1)+1$$
Therefore $\dim(\mathrm{K}_r(B^3,(N)_n)) =  \dim(S(B^3,(N)_n)) - \dim(\mathrm{V}_{2r}(S^2,(N)_n)) = n-2 $. Using the exact same method, we can prove that 
 $$ \dim(\mathrm{K}_r(B^3,(N)_{n-2},(2N-2))) =c(Nn-2N,N,n-2) =1$$
\end{proof}

\subsection{The action of the mapping class group} \label{action_MCG}
We denote by $\mathrm{M}(0,n)$ the mapping class group of the sphere with $n$ marked points, it is the group of orientation preserving diffeomorphisms of the sphere which globally preserve the $n$ marked points quotiented by the orientation preserving diffeomorphisms isotopic to the identity. Recall that $\mathrm{M}(0,n)$ is generated by $n-1$ elements $\sigma_1,...,\sigma_{n-1}$ where $\sigma_i$ is the half twist around the $i$-th and the $(i+1)$-st hole. Moreover the following relations give a presentation of $\mathrm{M}(0,n)$ (see \cite{Birman}) :
\begin{align}
  \sigma_i \sigma_j &= \sigma_j \sigma_i     & \text{when} \quad  \mid i-j \mid > 1 \label{R1}\\
 \sigma_i \sigma_j \sigma_i &= \sigma_j \sigma_i \sigma_j   &\text{when}  \quad \mid i-j \mid =1  \label{R2} \\
 \sigma_1 ... \sigma_{n-1} \sigma_{n-1} ...  \sigma_1 &= 1 \label{R3} \\
 (\sigma_1 ... \sigma_{n-1})^n  &=1  \label{R4}
\end{align}  For $i = 1,...,n-1$ we define $\phi_i$ : the decorated cobordism $S^2 \times I$ equipped with the colored banded arcs 
$$\includegraphics[scale = 0.26]{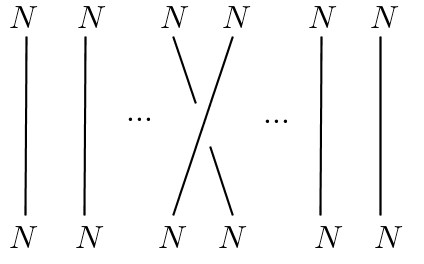}$$ where all banded arcs are colored by $N$, the half twist is done between the $i$-th and the $(i+1)$-st banded points. This decorated cobordism defines naturally an operator $Z(\phi_i)$ on $S(B^3,(N)_n)$. We can now define a projective representation $ \rho : \mathrm{M}(0,n) \rightarrow \mathrm{PAut}(S(B^3,(N)_n)) $ by $$ \rho(\sigma_i) = (-A)^{N(N+2)} Z(\phi_i)$$We denote by $\sigma_n$ the half twist between the $n$-th hole and the first hole. We make a similar definition for $\rho(\sigma_n)$ : let  $\phi_n$ be the decorated cobordism 
$$  \includegraphics[scale = 0.24]{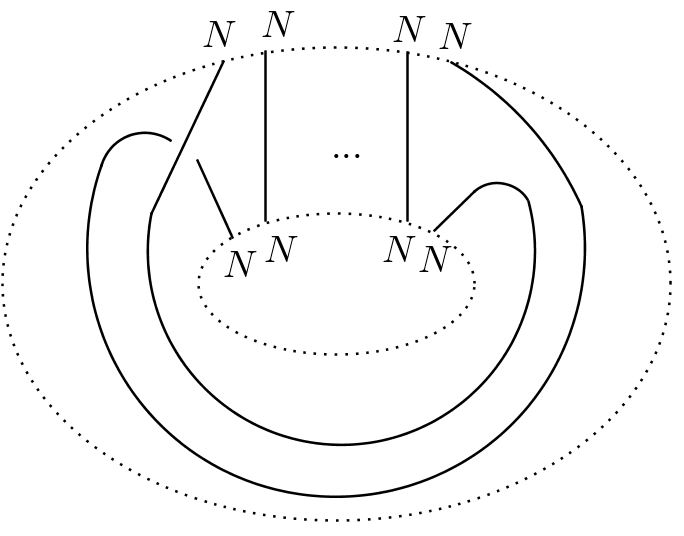}$$ where the underling manifold is drawn for more clarity. And we set  $ \rho(\sigma_n) = (-A)^{N(N+2)} Z(\phi_n)$.

\begin{remark} For $r \geq 2$, the space $\mathrm{K}_r(B^3,(N)_n)$ is known to be stable for the action of $\mathrm{M}(0,n)$ (see \cite[Prop 1.9]{BHMV}).
\end{remark}

\begin{remark} \label{R} With this choice of normalization, $\rho$ is only a projective representation because Relation (\ref{R3}) and (\ref{R4}) are only preserved by $\rho$ up to a scalar factor. But on the other hand, using isotopies, it is easy to see that the relations (\ref{R1}) and (\ref{R2}) remain true when applying $\rho$.
\end{remark}

\section{ Representations from cyclic branched covering of the sphere}
In this section we follow McMullen (see $\cite{mc}$). We will recall the definition of $\rho_q : \mathrm{M}(0,n)  \to \mathrm{PAut}(\mathrm{H}^1(X)_q )$.
\subsection{General definitions}
Let $d \geq 2$, $n \geq 2$ and $x_1,...,x_n$ $n$ distinct points in $\mathbb{C}$ fixed in this section. Consider $$X^{*} = \{ (x,y) \in \mathbb{C} \times \mathbb{C} \mid y^d = (x-x_1)...(x-x_n) \}$$ The first coordinate projection gives a degree $d$ ramified cover $ \pi : X^{*} \rightarrow \mathbb{C}$ ramified over $x_1,..., x_n$. We denote by $\hat{\mathbb{C}} =  \mathbb{C} \cup \{ \infty \}$ the standard compactification of $\mathbb{C}$. The cover $\pi : X^{*} \rightarrow \mathbb{C}$ extends to a unique cover $ \pi : X \rightarrow \hat{\mathbb{C}}$ where $X$ is a compact orientable surface without boundary. The genus of $X$ is given by the Riemann-Hurwitz formula : $ g(X) = \frac{(n-1)(d-1)+1-\text{gcd}(d,n)}{2}$.

 The group of automorphisms of the cover $X$ is isomorphic to $ \mathbb{Z}/d$ and generated by the diffeomorphism  $T : (x,y) \in X \mapsto (e^{2i \pi/ d}x,y)\in X$.  The space $\mathrm{H}^1 (X, \mathbb{C})$ splits as $$ \mathrm{H}^1(X,\mathbb{C}) = \bigoplus_{q^d = 1} \mathrm{H}^1(X)_{q,d}$$ where  $\mathrm{H}^1(X)_{q,d} = \ker(T^{*} - q \mathrm{Id})$.
 
  Let $\text{Mod}(X)^{T}$ be the group of orientation preserving diffeomorphisms up to isotopies of $X$ which commute with $T$. Let $\text{Mod}_{c} (\mathbb{C} ,\{x_1,...,x_n \})$ be the group of compactly supported orientation preserving diffeomorphisms fixing a neighborhood of $\infty$ up to isotopies of $\mathbb{C}$ which preserve the set of points $\{x_1,...,x_n \}$. It is known that $\text{Mod}_{c} (\mathbb{C} ,\{x_1,...,x_n \})$ is isomorphic to the braid  group $B_n$. 

Now if $\phi \in \text{Mod}_{c} (\mathbb{C} ,\{x_1,...,x_n \})$, there is a unique lift $\tilde{\phi} \in \text{Mod}(X)^{T}$ which fixes a neighborhood of $\pi^{-1}( \infty )$. So we have a representation $$\rho_{q,d} : B_n \rightarrow  \text{Aut}(\mathrm{H}^1(X)_{q,d})$$ Actually the representation $\rho_{q,d}$ depends only on $q$ (see \cite[Remark: stabilization]{mc}), more precisely if $d \neq d'$ and $q^d = q^{d'}=1$ then $\dim( \mathrm{H}^1(X)_{q,d}) = \dim( \mathrm{H}^1(X)_{q,d'})$ and 
$ \rho_{q,d} \simeq \rho_{q,d'}$ so we simply write $\rho_q$ for $\rho_{q,d}$ and $ \mathrm{H}^1(X)_{q}$ for  $\mathrm{H}^1(X)_{q,d}$.

\begin{remark} Recall that $\mathrm{M}(0,n)$ is a quotient of $B_n$.  McMullen proves that the cardinality of $\pi^{-1}(\infty) $ is $\text{gcd}(d,n)$. Hence the projective representation $\rho_q$ is well defined on $\mathrm{M}(0,n)$ if and only if the point $\infty$ is not ramified, in other words if and only if $q^n=1$.
\end{remark}

\paragraph{Theorem (see \cite[Corollary 3.3]{mc})} Let $q$ be a root of unity different from $1$ and $n$ be the number of marked points. One has
\begin{align}
  \dim(\mathrm{H}^1(X)_q) &= n-1  \quad \quad \text{when} \quad  q^n \neq 1 \notag \\
\dim(\mathrm{H}^1(X)_q) &= n-2   \quad \quad \text{when}  \quad q^n = 1  \notag 
\end{align} 
and $\mathrm{H}^1(X)_1 = 0$.

\begin{remark} Let $N \geq 1$, $r \geq 2$ such that $2r = Nn$ and $A_r$ be a $4r$-th primitive  root of unity. Set $q = A_r^{4N}$. We have that $q^n=1$ hence $\mathrm{K}_r(B^3,(N)_n))$ and $\mathrm{H}^1(X)_{q^{-1}}$ have the same dimension (which is $n-2$). Moreover $\mathrm{M}(0,n)$ acts projectively on both of them.
\end{remark}

\begin{remark} \label{remark mc}
In \cite[Theorem 5.5]{mc}, McMullen proves that when $q^n \neq 1$, the representation $\rho_q$ is dual to the reduced Burau representation specialized to $t=q$. But for $q^n =1$ this is not the case (because the dimension of $\mathrm{H}^1(X)_q$ is $n-2$ and not $n-1$).

\end{remark}

In \cite{mc} it is also shown that the space $\mathrm{H}^1(X)_q$ is endowed with a natural hermitian form ($\mathbb{C}$-antilinear on the right and $\mathbb{C}$-linear on the left) which we denote by $\langle \,  ,  \rangle$ \footnote{This form is defined by $\langle \, \alpha ,  \beta \, \rangle =i/2 \int \alpha \wedge \overline{\beta}$.} (the dependence in $q$ is implicit). The representation $\rho_q$ is a unitary projective representation with respect to this form.
\subsection{Explicit formulas} \label{Explicit formulas}
In this subsection we suppose $q^n =1 $ and $q \neq 1$. We want to give a basis of $\mathrm{H}^1(X)_q$ and explicit formulas for the action of the braid group. McMullen already gave a spanning set and explicit formula for $\rho_q$ in terms of this spanning set.
\paragraph{Theorem (see \cite[Theorem 4.1]{mc})} There is a spanning set $(u_j)_{1}^{n}$ for $\mathrm{H}^1(X)_q$ such that the hermitian form is given by 
\begin{align}
 \langle u_j, u_j \rangle  &= -i ( q-\overline{q}) \notag \\
 \langle u_j, u_{j+1} \rangle  &= i (1- \overline{q}), \, \text{and} \notag \\
 \langle u_j, u_{k} \rangle  &= 0 \quad \text{if} \mid j-k \mid > 1 \notag
\end{align} 
(the indices $j,k \in \mathbb{Z}/n$ and $\mid j-k \mid > 1$ means $j-k \neq -1,0$ or $1 \, \text{mod} \, n$). The action of the braid group is given by $$ \rho_q(\sigma_j) (x) = x- \frac{i}{2} \langle x, u_j \rangle u_j$$ if $q= -1$, and otherwise by $ \rho_q(\sigma_j) (x) = x- (q+1) \frac{\langle x, u_j \rangle}{ \langle u_j, u_j \rangle } u_j$
\\

Let $q \neq 1$ such that $q^n =1$. We chose a square root of $q$ which we denote by $q^{1/2}$ such that $i(q^{1/2}- q^{-1/2}) \leq 0$. Let $\Delta$ be a square root of $i(q^{1/2}- q^{-1/2})$, we note that $\Delta$ is a pure imaginary number. For $j=1,...,n$ we define $$ \tilde{u}_j = \frac{(q^{-1/2})^j u_j}{\Delta} $$ In what follows we don't need $\tilde{u}_n$. We can give a basis of $\mathrm{H}^1(X)_q$.

\begin{proposition} \label{prop2_c3} The set $(u_j)_{1}^{n-2}$ or equivalently the set $(\tilde{u}_j)_{1}^{n-2}$ is a basis of $\mathrm{H}^1(X)_q$ and in this basis $$\tilde{u}_{n-1} = e_1(\delta) \tilde{u}_{n-2}-e_2(\delta) \tilde{u}_{n-3}+...+ (-1)^{n-1} e_{n-2}(\delta) \tilde{u}_1$$ with $\delta = -q^{1/2}-q^{-1/2}$.
\end{proposition}
\begin{proof}
A straightforward computation gives
that the determinant of the matrix $M = ( \langle \tilde{u}_j , \tilde{u}_k \rangle)_{j,k = 1,...,n-2}$ can be expressed using the $(n-2)$-th Chebyshev polynomial : 
$$\det(M)  =(-1)^{n-2} e_{n-2}(\delta ) = e_{n-2}(q^{1/2}+q^{-1/2}) =   \frac{q^{(n-1)/2}-q^{-(n-1)/2}}{q^{1/2}-q^{-1/2}} \neq 0 $$ It is not zero because $q^n=1$  and $q \neq 1$. Therefore since the sesquilinear form $\langle \, , \rangle$ is non degenerate and since $\dim(\mathrm{H}^1(X)_q) = n-2$, we conclude that $(\tilde{u}_j)_{1}^{n-2}$ is a basis. The formula for $\tilde{u}_{n-1}$ comes from a direct computation.

\end{proof}

In the basis $(\tilde{u}_j)_{1}^{n-2}$ we have the following expression of $\rho_q(\sigma_1),..., \rho_q(\sigma_{n-1})$ 

 \paragraph{Conclusion 1} For $k = 1,...,n-2$ and $j = 1,...n-1$  by \cite[Theorem 4.1]{mc}
\begin{align}
\rho_q(\sigma_j)(\tilde{u}_k)&= -q \tilde{u}_k &  \text{when} \quad j=k \quad \quad \, \,  \, \, \, \notag \\
\rho_q(\sigma_j)(\tilde{u}_k)&= \tilde{u}_k + q^{1/2} \tilde{u}_j &  \text{when} \quad \mid j-k \mid = 1 \notag \\
\rho_q(\sigma_j)(\tilde{u}_k)&=  \tilde{u}_k &  \text{when} \quad \mid j-k \mid  > 1 \notag 
\end{align} 
with $\tilde{u}_{n-1} = e_1(\delta) \tilde{u}_{n-2}-e_2(\delta) \tilde{u}_{n-3}+...+ (-1)^{n-1} e_{n-2}(\delta) \tilde{u}_1$.

\section{Proof of Theorem \ref{main_theorem_punctured_spheres} for $n \geq 6$} \label{proof_n_6}
 As we will see later on, we have to separate the case $n=4$ and the case $ n \geq 6$ (see Remark \ref{r1}). In this section we deal with the case $n \geq 6$. Suppose that $n = 2n'\geq 6$. Let $A_r$ be a $4r$-th primitive root of unity with $ r = n' N$. 

 One good way to see  elements in the kernel space $\mathrm{K}_r(B^3,(N)_n)$ is to consider banded trivalent colored graphs in the $3$-ball with a coloring which admissible but not $2r$-admissible. We recall that an admissible (resp. $2r$-admissible) coloring of a uni-trivalent graph $G$ is an assignment of colors to the edges of $G$ such that at every trivalent vertex of $G$, the triple of colors is admissible (resp. $2r$-admissible), where admissible is defined as follows :
\begin{definition}
     For $0 \leq a, b, c $ three integers the triple $(a,b,c)$ is said admissible when 
\begin{align}
  a+b+c  &  \equiv 0 \, \pmod 2    \notag \\
  \mid a-c \mid \leq & \, b \leq a+c  \notag 
\end{align}
Moreover the triple $(a,b,c)$ is said $2r$-admissible if it is admissible and if 
$$ 0 \leq  \,  a, b, \,  c \leq r -2 \quad \quad \text{and}\quad \quad  a+b+c  \leq 2r-4  $$
\end{definition}

Now let us consider in $S_{A_r}(B^3,(N)_{n})$ the element 
 $$u \quad = \quad \begin{minipage}[c]{5cm}
\includegraphics[scale = 0.28]{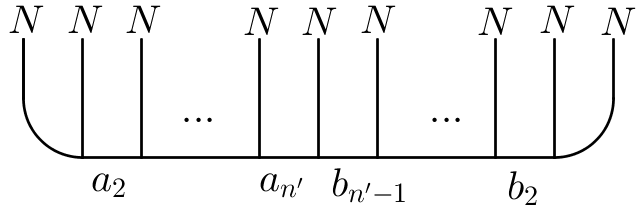}
\end{minipage}$$ where $a_j = j N-2$ and $b_j = j N$ for all $j \in \mathbb{Z}$. This is a colored banded graph in the $3$-ball and the univalent vertices of the graph are attached to the banded points on the sphere.
 Observe that the colorings of this graph are smaller or equal to $ r-1 = n' N -1$. Therefore the coloring of the graph makes sense and the vector $u$ exists in $S_{A_r}(B^3,(N)_{n})$. Moreover this graph is clearly a dual graph of a certain pants decomposition of the $n$-th punctured sphere and the coloring of this graph is admissible so the vector $u$ is not zero in $S_{A_r}(B^3,(N)_{n})$.

  Now we see that the sum of the colors of the edges meeting  at the vertex "under" the $(n'+1)$-st puncture is $$ N +a_{n'} +b_{n'-1}= 2 n' N-2  = 2r-2 > 2r-4 $$ So the triple $(a_{n'},b_{n'-1},N)$ is not $2r$-admissible and  $u \in \mathrm{K}_r(B^3,(N)_n)$. To avoid any confusion we give an example
  \paragraph{Example when $n=10$ and $N=4$.} In this case $r = 20$ and $$u \quad =  \begin{minipage}[c]{4cm}
\includegraphics[scale = 0.26]{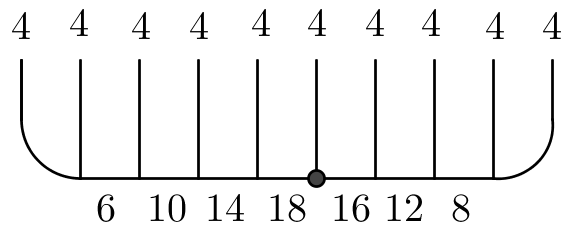}
\end{minipage}$$ we see that at the vertex under the $6$-th puncture (this vertex is shown with a dot) $18+4+16 = 38 > 2r-4 = 36$ so $u \in \mathrm{K}_{20}(B^3,(4)_{10})$.
\\

   \noindent Now we consider the following  decorated cobordisms
 $$ S  =  \begin{minipage}[c]{2.5cm}
\includegraphics[scale = 0.23]{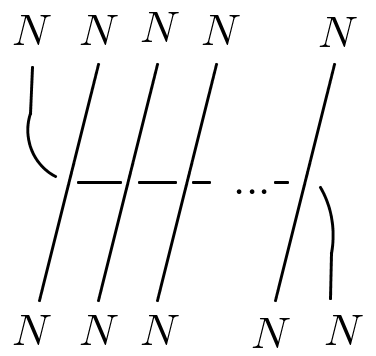}
\end{minipage}\quad \text{and} \quad \quad  S_0 = \begin{minipage}[c]{2.5cm}
\includegraphics[scale = 0.23]{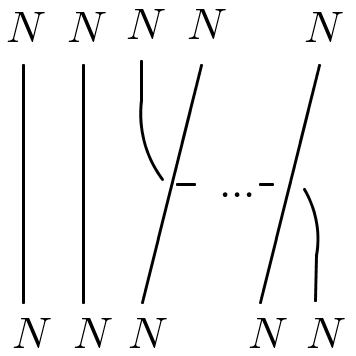}
\end{minipage}$$ They give operators which we denote by $s, \, s_0 \in \text{End}(S_{A_r}(B^3,(N)_n))$. These operators are invertible and for later use we remark using isotopies that 
\begin{align}
s  \rho(\sigma_j) s^{-1}& = \rho(\sigma_{j+1})  & \forall j = 1,...,n  \notag  \\
s_0  \rho(\sigma_j) s_0^{-1}& = \rho(\sigma_{j+1})  & \forall j = 3,...,n-2   \notag
\end{align}
with the convention $\sigma_{n+1} = \sigma_1$. For later use we state the following lemma which can be proved using elementary skein theoretical methods.
\begin{lemma}
 In $S_{A_r}(B^3,(N)_3,(3N-2)) $ which is $2$ dimensional we have  \begin{equation} \label{eq1_c3} \begin{minipage}[c]{2cm}
\includegraphics[scale = 0.15]{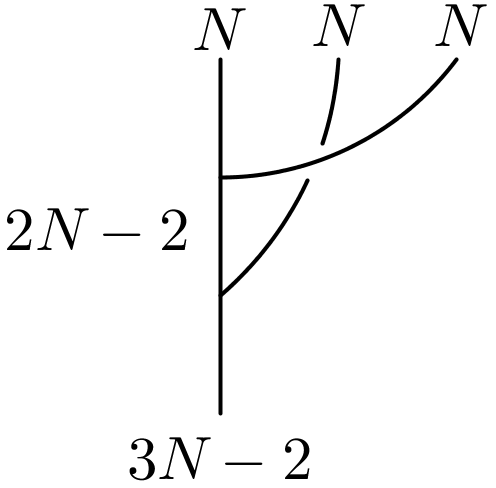} 
\end{minipage} = \quad A_r^{N^2} \, \begin{minipage}[c]{2.2cm}
\includegraphics[scale = 0.15]{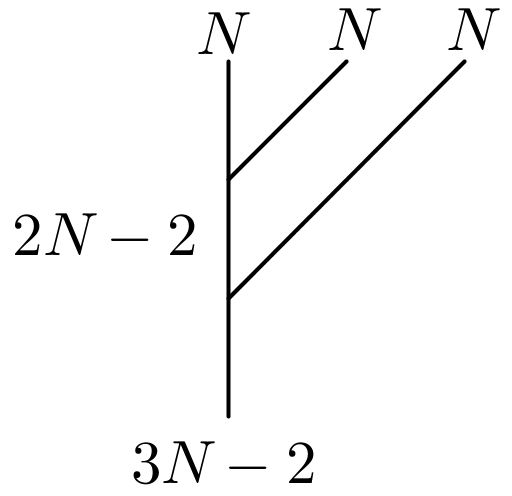} 
\end{minipage}   + \, A_{r}^{N^2-2N} \begin{minipage}[c]{3.7cm}
\includegraphics[scale = 0.15]{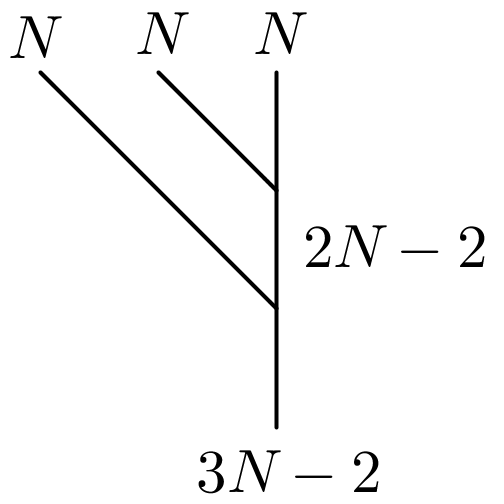} 
\end{minipage} \end{equation} 
\end{lemma}

\begin{remark} The first step to understand the action of $\mathrm{M}(0,n)$ on $\mathrm{K}_r(B^3,(N)_n)$ is to compute the actions of the generators $(\sigma_j)$ on the vector $u$. Notice that in the general skein module, this computation can be made using well known techniques such as fusion rules (see \cite{MV}) but these techniques produce complicated expressions. Surprisingly, the expressions we get are quite simple since when we work in $\mathrm{K}_r(B^3,(N)_n)$ (see Proposition \ref{l2_c3}). These simplifications are mainly obtained using $(\ref{eq1_c3})$ and using that $ \dim(\mathrm{K}_r(B^3,(N)_{n-2},(2N-2))) = 1 $ (see Theorem \ref{cor3_c3}).
\end{remark}

\begin{proposition} \label{l2_c3}
There exists $\lambda \in \mathbb{C}-\{0 \}$ such that 
\begin{align}
\rho(\sigma_1)(u) & = (-1)^{N-1} A_r^{2N(N-1)} u \notag \\
 \rho(\sigma_2)(u) &= (-1)^{N} A_r^{2N(N+1)} u + (-1)^N A_r^{2N^2} \lambda s(u)  \notag \\
\rho(\sigma_j)(u) & = (-1)^{N} A_r^{2N(N+1)} u  \notag \\
 \rho(\sigma_n)(u) &= (-1)^{N} A_r^{2N(N+1)} u + (-1)^N A_r^{2N^2} \lambda^{-1} s^{-1} (u)  \notag
\end{align} for $ j = 3 ,...,n-1$.
\end{proposition}

\begin{proof}
We recall the following (see \cite{MV}) : if $0 \leq a,b, c \leq r-1$ is a admissible triple then 
$$ \begin{minipage}[c]{1cm}
\includegraphics[scale = 0.25]{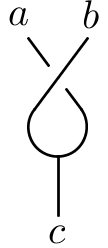} 
\end{minipage}= \quad  A_{r}^{ij-k(i+j+k+2)} \begin{minipage}[c]{1cm}
\includegraphics[scale = 0.25]{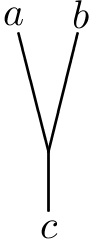} 
\end{minipage}$$ where $i = (b+c-a)/2$, $ j = (a+c-b)/2$ and $ k= (a+b-c)/2$. From this, it is clear that  
$\rho(\sigma_1)(u)  = (-1)^{N-1} A_r^{2N(N-1)} u$ and $\rho(\sigma_{n-1})(u) = (-1)^{N} A_r^{2N(N+1)} u$ since $a_2 = 2N-2$ and $b_2 = 2N$.

\noindent  Note that $ \begin{minipage}[c]{2.8cm}
\includegraphics[scale = 0.15]{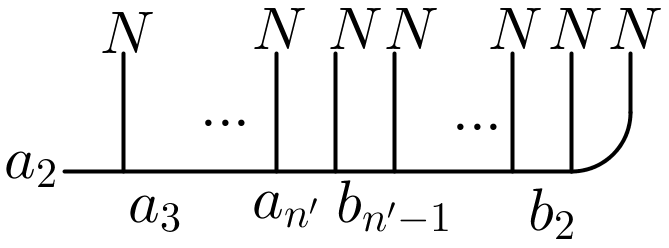} 
\end{minipage}$ belongs to $\mathrm{K}_r(B^3, (N)_{n-2},(2N-2))$ which one dimensional by Theorem \ref{cor3_c3}. Hence for $ 3 \leq j \leq n-2$ we have $\rho(\sigma_{j})(u), s_0(u) \in \mathbb{C} u$. Using $s_0  \rho(\sigma_j) s_0^{-1} = \rho(\sigma_{j+1})$ and $\rho(\sigma_{n-1})(u) = (-1)^{N} A_r^{2N(N+1)} u$ we conclude that $$ \rho(\sigma_j)(u)  = (-1)^{N} A_r^{2N(N+1)} u $$ Let us now compute $\rho(\sigma_2)(u)$. Using $(\ref{eq1_c3})$ we have
$$  
 \rho(\sigma_2)(u) = (-1)^{N} A_r^{2N(N+1)} u + (-1)^N A_r^{2N^2} v 
\quad \text{and} \quad 
  \rho(\sigma_1)(v) = (-1)^{N} A_r^{2N(N+1)} v + (-1)^N A_r^{2N^2} u $$
where  $$v  =  \begin{minipage}[c]{5.5cm}
\includegraphics[scale = 0.27]{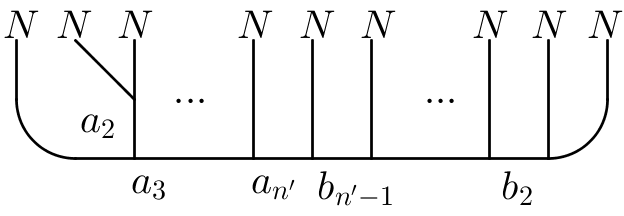}
\end{minipage}\quad $$

\noindent We would like to compare $v$ with $s(u)$. We recall that :  \begin{align}
  s(u)  & =  \begin{minipage}[c]{5cm}
\includegraphics[scale = 0.25]{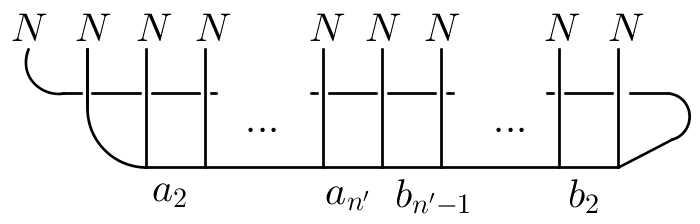}
\end{minipage}  =  \begin{minipage}[c]{5cm}
\includegraphics[scale = 0.25]{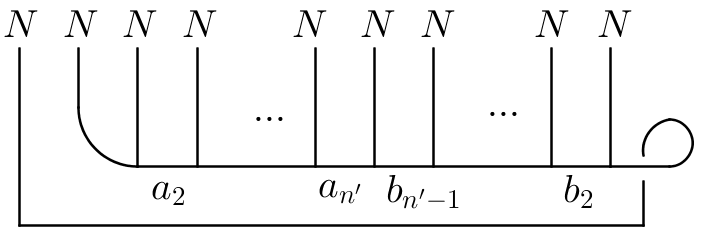}
\end{minipage} \notag \\
& = (-A_r)^{-(N+2)N}  \begin{minipage}[c]{7cm}
\includegraphics[scale = 0.25]{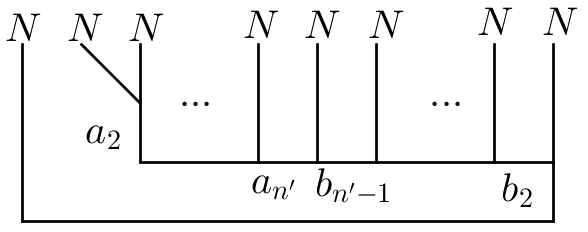}
\end{minipage} \notag \end{align}
 Now notice that $\, \, \begin{minipage}[c]{4.3cm}
\includegraphics[scale = 0.27]{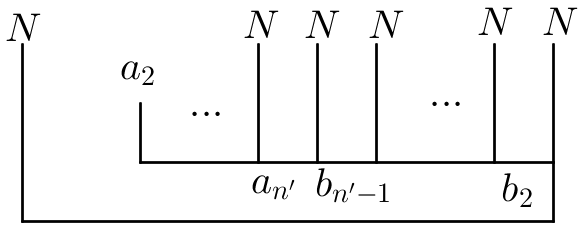}
\end{minipage} $ and $ \, \, \begin{minipage}[c]{4.5cm}
\includegraphics[scale = 0.27]{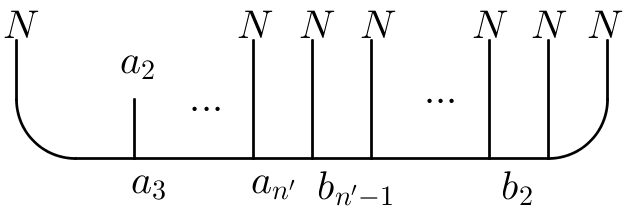}
\end{minipage}$ both belong to  the one dimensional space $\mathrm{K}_r(B^3, (N)_{n-2},(2N-2))$. They are proportional to each other, therefore  $v$ and $s(u)$ are proportional. Let $\lambda$ be the  non zero complex number such that $v = \lambda s(u)$. So far we have proved that \begin{equation}  \label{eq2_c3}
\rho(\sigma_2)(u) = (-1)^{N} A_r^{2N(N+1)} u + (-1)^N A_r^{2N^2} \lambda s(u)
\end{equation} and \begin{equation} \label{eq3_c3}
 \rho(\sigma_1)(\lambda s(u)) = (-1)^{N} A_r^{2N(N+1)}\lambda s(u) + (-1)^N A_r^{2N^2} u
\end{equation}
Now using $ \rho(\sigma_n)= s^{-1} \rho(\sigma_1) s $ and  (\ref{eq3_c3}) we have
$$
\lambda \rho(\sigma_n)( u) = s^{-1} \rho(\sigma_1)  \left( \lambda s( u) \right) = (-1)^N A_r^{2N(N+1)} \lambda u + (-1)^N A_r^{2N^2}  s^{-1} (u)$$ From which we get the last equality claimed by the Proposition.

\end{proof}

\begin{remark} \label{r1} Note that, in the proof of Proposition \ref{l2_c3}, the way that (\ref{eq1_c3}) is used only makes sense when $n > 4$ or for $n=4$ and $N=1$ which is the case treated in \cite{AMU}. This is why the case $n=4$ and $N \geq 2$ has to be handled separately.
\end{remark}
To simplify the formulas we set  $$q=A_r^{4N} \quad \text{and} \quad \chi_0= (-1)^{N-1} A_r^{2N(N-1)} $$
The next lemma tells us that Proposition \ref{l2_c3} gives the action of $\mathrm{M}(0,n)$ on the set of vectors $(w_j)$, which are defined for $j=1,...,n$ by $$w_j = \lambda^{j-1} s^{j-1}(u) \in \mathrm{K}_r(B^3,(N)_n)$$ 
\begin{lemma} \label{l3_c3}
For $j=1,...,n-1$ and $k= 1,...,n-1$ 
\begin{align}
\rho(\sigma_j)(w_j) & = \chi_0  w_j &  \notag \\
\rho(\sigma_j)(w_k) & = \chi_0 (-q w_k - q^{1/2} w_j) & \text{when} \,  \mid j-k \mid =1 \notag \\
\rho(\sigma_j)(w_k) & = \chi_0 (-q w_k)& \text{when} \,  \mid j-k \mid >1 \notag  
\end{align}
\end{lemma}

\begin{proof}
$u = w_1$ so from  Proposition \ref{l2_c3} we have 
$$
\rho(\sigma_1)(w_1)  = \chi_0  w_1 \quad \text{and} \quad 
\rho(\sigma_1)(w_2) = \chi_0 (-q w_2 - q^{1/2} w_1)$$  Now if $j = 3,...,n-1$, 
$$
\rho(\sigma_1)(w_j)  =\sigma_1 \lambda^{j-1} s^{j-1} (u) = \lambda^{j-1} s^{j-1} \left( s^{-(j-1)} \sigma_1  s^{j-1} \right) (u) $$ But $s^{-(j-1)} \sigma_1  s^{j-1} = \sigma_{n-j+2}$ so by Proposition \ref{l2_c3} $$\rho(\sigma_1)(w_j) = \lambda^{j-1} s^{j-1} \sigma_{n-j+2} (u)  = \lambda^{j-1} s^{j-1} \chi_0 (-q u) = \chi_0 (-q w_j)$$ so the lemma is true for $\sigma_1$. By induction we can continue the same method and we get what is claimed by the lemma.

\end{proof}


\begin{proposition} \label{prop3_c3} The set  $(w_j)_{1}^{n-2}$ is a basis of $\mathrm{K}_r(B^3,(N)_n)$ and in this basis $$w_{n-1} = e_1(\delta) w_{n-2}-e_2(\delta) w_{n-3}+...+ (-1)^{n-1} e_{n-2}(\delta) w_1$$ with $\delta = -q^{1/2}-q^{-1/2}$.
\end{proposition}
 
\begin{proof}
To prove that $(w_j)_{1}^{n-2}$ is a basis, it is enough to prove that this family is linearly independent since $\dim(\mathrm{K}_r(B^3,(N)_n))=n-2$. Let $\beta_1,...,\beta_{n-2} \, \in \mathbb{C}$ such that $$\sum_{l=1}^{n-2} \beta_l \, w_l =0$$ Let $ 1 \leq j \leq n-2$. We have $\left( \rho(\sigma_j)+q \, \mathrm{Id} \right) \left(\sum_{l=1}^{n-2} \beta_l \, w_l \right) =0$ which gives by Lemma \ref{l3_c3} $$
((1+q) \beta_j-q^{1/2} \beta_{j-1} - q^{1/2} \beta_{j+1}) w_j =0 $$ with the convention $\beta_{-1}= \beta_{n-1} = 0$. Since $w_j \neq 0$ we have $ - \beta_{j-1}+(q^{1/2}+q^{-1/2}) \beta_j -  \beta_{j+1} =0 $. Hence $M \overrightarrow{\beta} = 0$ where 
 $$ M = \left( \begin{array}{cccc}
-\delta &- 1  &    & 0 \\
-1 & -\delta &  \ddots &  \\
 &  \ddots & \ddots  & -1 \\
0   &   & -1 & - \delta \\
\end{array} \right) \quad \text{and} \quad  \overrightarrow{\beta} =  \left( \begin{array}{c}
\beta_1 \\
\vdots \\
\beta_{n-2} \\\end{array} \right)  $$ with $ \delta = -q^{1/2}-q^{-1/2}$.   
We have already seen that $$\det(M) = (-1)^{n-2} e_{n-2}(\delta)= \frac{q^{(n-1)/2}-q^{-(n-1)/2}}{q^{1/2}-q^{-1/2}} \neq 0$$ since $q = A_r^{4N}$ with $A_r$ a $2Nn$-th primitive root of unity. So $\beta_1 = \, ... \,  = \beta_{n-2} =0$ and $w_1,...,w_{n-2}$ are linearly independent. The expression of $w_{n-1}$ is obtained by a direct computation.

\end{proof}
 Multiplying by $(-\chi_0 q)^{-1}$ all equations in Lemma \ref{l3_c3}, we obtain the following conclusion :
  
  \paragraph{Conclusion 2}  For $k = 1,...,n-2$ and $j = 1,...n-1$ 
\begin{align}
(-\chi_0 q)^{-1} \rho(\sigma_j)(w_k)&= -q^{-1} w_k &  \text{when} \quad j=k   \quad \quad \, \,  \, \, \,  \notag \\
(-\chi_0 q)^{-1}\rho(\sigma_j)(w_k)&=  w_k +q^{-1/2} w_j &  \text{when} \quad \mid j-k \mid = 1 \notag \\
(-\chi_0 q)^{-1} \rho(\sigma_j)(w_k)&= w_k &  \text{when} \quad \mid j-k \mid  > 1 \notag
\end{align}
with $w_{n-1} = e_1(\delta) w_{n-2}-e_2(\delta) w_{n-3}+...+ (-1)^{n-1} e_{n-2}(\delta) w_1$ (using Proposition \ref{prop3_c3}).

\begin{proof}[\textbf{Proof of Theorem \ref{main_theorem_punctured_spheres}}] By Proposition \ref{prop3_c3}, $(w_j)_{j=1,...,n-2}$ is basis of $\mathrm{K}_r(B^3,(N)_n)$. Therefore, by comparing Conclusion 1 with Conclusion 2 we deduce that $$\mathrm{K}_r(B^3,(N)_n)\,  \simeq \,   \mathrm{H}^{1}(X)_{q^{-1}}$$ as projective representations of $ \mathrm{M}(0,n)$ 
\end{proof}

\section{Proof of Theorem \ref{main_theorem_punctured_spheres} when $n=4$} \label{proof_n_4}
\subsection{Action of $\mathrm{M}(0,4)$ on $S(B^3,(N)_4)$}
In this subsection, we show how to produce explicit formulas for the action of $\mathrm{M}(0,4)$ on $S(B^3,(N)_4)$.
Let $0 \leq k \leq N$, we define $Y^k X^{N-k} \in S(B^3,(N)_4)$ by:
$$Y^k X^{N-k} =
\begin{minipage}[c]{3cm}
\includegraphics[scale = 0.18]{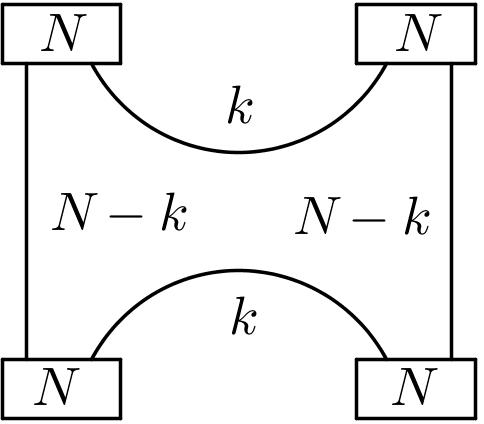}
\end{minipage}$$
where the lines represent colored bands in the $3$-ball. It is easy to see that $(Y^k X^{N-k} )_k$ is a basis of $S(B^3,(N)_4)$. Recall that $\rho : \mathrm{M}(0,4) \to \mathrm{PAut} (S(B^3,(N)_4))$ was defined in Section \ref{action_MCG}. We recall the following rule for all $L \geq j \geq 1$ : \begin{equation} \label{resolution} \begin{minipage}[c]{2cm} \includegraphics[scale = 0.09]{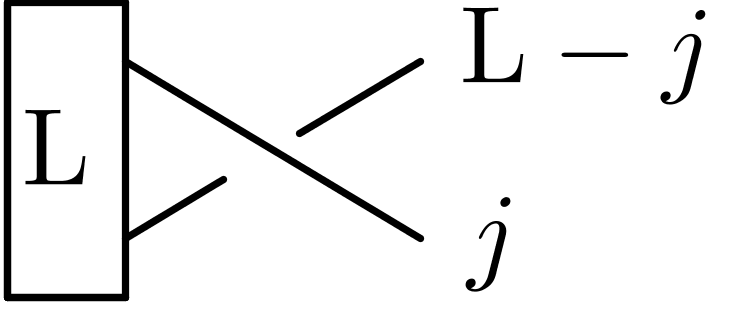}
\end{minipage}  = A^{j(L-j)} \, \,\,   \begin{minipage}[c]{2.5cm} \includegraphics[scale = 0.09]{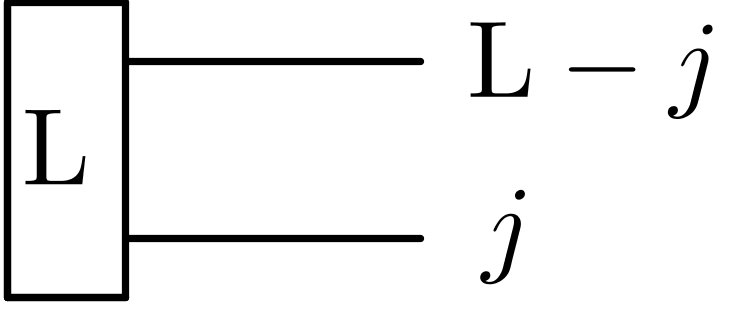}
\end{minipage} \end{equation}
\begin{lemma} \label{l4_c3}
 For $0  \leq k \leq N$, we have :
$$ \rho(\sigma_3) Y^k X^{N-k} = \rho(\sigma_1) Y^k X^{N-k} = (-A)^{k(k+2)}
\begin{minipage}[c]{4cm}
\includegraphics[scale = 0.18]{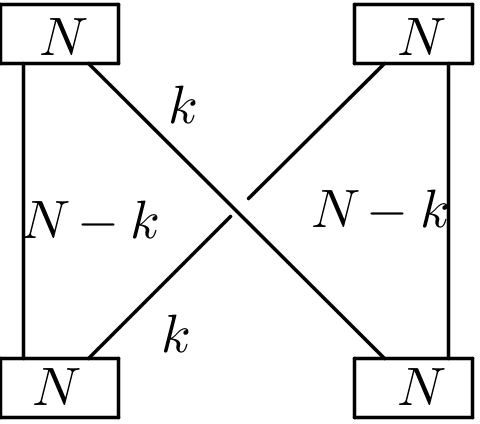}
\end{minipage}  $$
and $ s \circ \rho(\sigma_2) \circ s = \rho(\sigma_1)$
where $s$ is defined by : $ s(Y^k X^{N-k}) = Y^{N-k} X^{k} $.
\end{lemma}
\begin{proof} In this setting, $\sigma_1$ is the half twist between the box in the bottom left corner and the box in top left corner. Therefore using the definition given in Section \ref{action_MCG}, we have $$
\rho(\sigma_1 )( Y^k X^{N-k}) =  (-A)^{N(N+2)} \quad  \begin{minipage}[c]{4cm} \includegraphics[scale = 0.2]{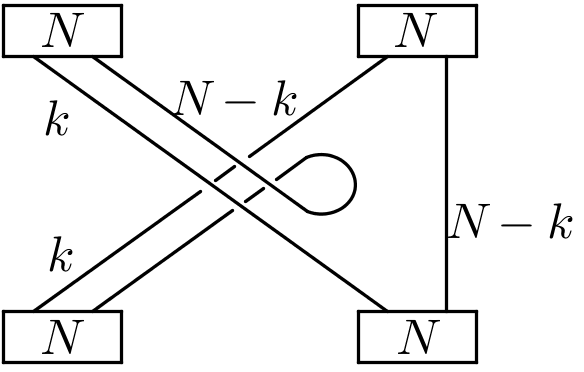}
\end{minipage}$$
we can now conclude using (\ref{resolution}) and
 $ \begin{minipage}[c]{1.1cm} \includegraphics[scale = 0.13]{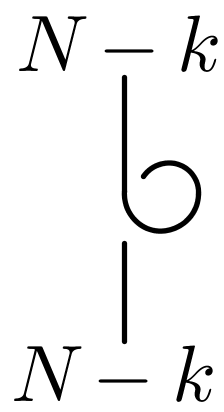}
\end{minipage}  = (-A)^{-(N-k)(N-k+2)}   \begin{minipage}[c]{0.8cm} \includegraphics[scale = 0.13]{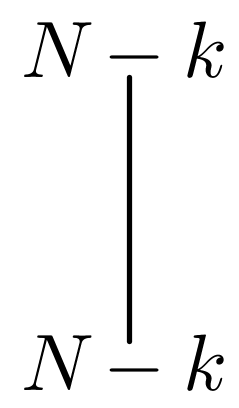}
\end{minipage} $.

\end{proof}
\begin{proposition} \label{prop4_c3} Let $L \geq 1$. One has  \begin{equation} \label{eq7_c3}   \begin{minipage}[c]{2.2cm}\includegraphics[scale = 0.15]{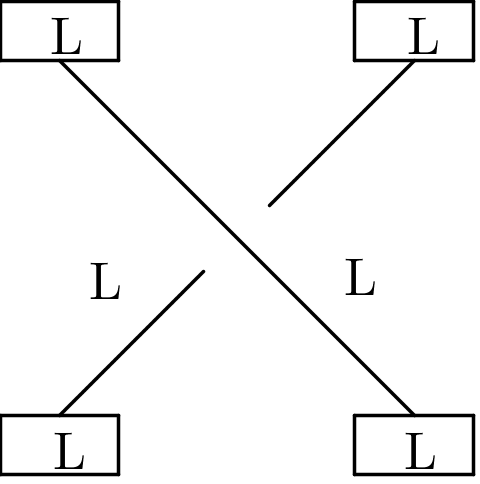}
\end{minipage}  = \quad \prod_{k=0}^{L-1} ( A^{2k+1}Y+A^{-2k-1}X) \end{equation}
where the product on the right hand side has to be understood as a formal notation and has to be developed using the rule $XY = YX$ to have the coefficients in the basis $(Y^k X^{N-k} )_k$.
\end{proposition}

\begin{proof} If $L=1$ it is a straightforward computation. Suppose now $L \geq 2$ and let us call $P_L(X,Y)$ the quantity on the left hand side of Equation (\ref{eq7_c3}). Resolving the crossing on the left corner, we have :
$$
  \begin{minipage}[c]{2cm} \includegraphics[scale = 0.15]{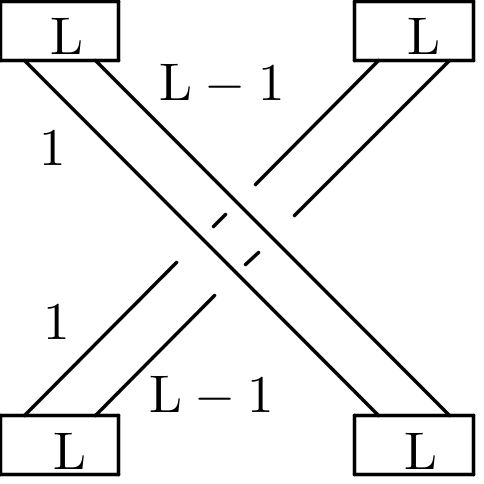}
\end{minipage} = A \quad  \begin{minipage}[c]{2cm} \includegraphics[scale = 0.15]{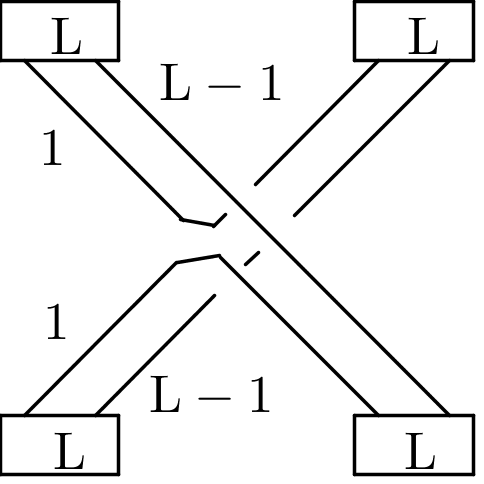}
\end{minipage} \, + A^{-1}  \begin{minipage}[c]{2cm}
\includegraphics[scale = 0.15]{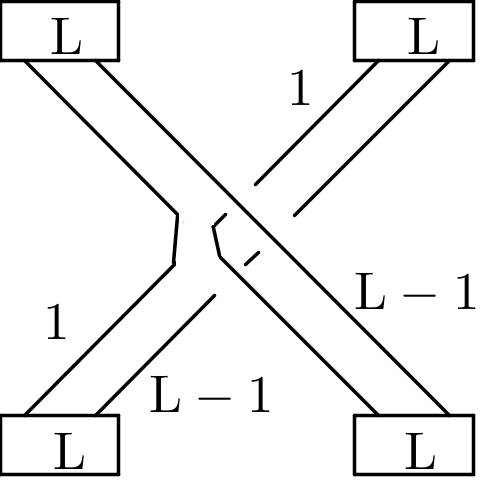}
\end{minipage}
=  \, A^{2L-1} \, \begin{minipage}[c]{2cm} \includegraphics[scale = 0.15]{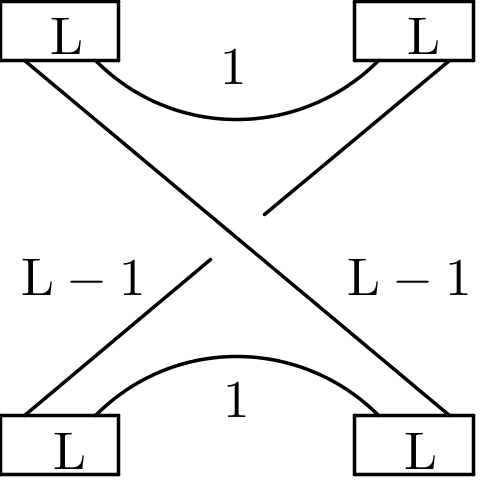}
\end{minipage}  +  A^{-2L+1} \,\begin{minipage}[c]{2.5cm}
\includegraphics[scale = 0.15]{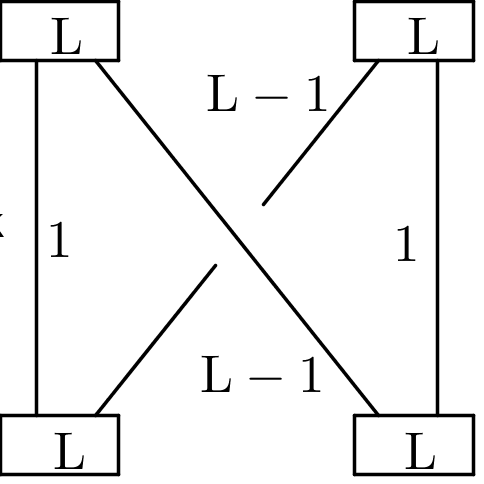}
\end{minipage} $$
where the last equality is obtained using (\ref{resolution}). This means : $$P_L(X,Y) = A^{2L-1}YP_{L-1}(X,Y) + A^{-2L+1}XP_{L-1}(X,Y)  = (A^{2L-1} Y + A^{-2L+1} X) P_{L-1}(X,Y)$$ The formula of the proposition is then given by induction on $L$.
\end{proof}

\begin{remark} \label{rem}We denote by $(M_{j,k})$ (resp. ($\overline{M}_{j,k}$)) the matrix of  $\rho(\sigma_1)$ (resp. $\rho(\sigma_2)$) in the basis $(Y^k X^{N-k} )_k$. We deduce from  (\ref{eq7_c3}) that the matrix $(M_{j,k})$ is upper triangular and the matrix  $(\overline{M}_{j,k})$ is lower triangular. Both have eigenvalues $c_j$ for $j= 0,...,N$ where $$c_j = (-1)^j A^{j(2j+2)}$$ Let  $A_r$ be a $8N$-th primitive root of unity. In Proposition \ref{prop6_c3} we will need that $M_{N-1,N}$ evaluated at $A=A_r$ is not zero. Developing (\ref{eq7_c3}) we can check that $$(M_{N-1,N})_{\mid_{A=A_r}} =2  (-1)^N \left( \frac{ A_r^{2N^2+2N} }{A_r^2-A_r^{-2}} \right)$$
\end{remark}
 
\begin{remark} After I finished the proof of Proposition \ref{prop4_c3}, I realized that Frenkel and Khovanov (see \cite[Prop~3.1.1]{FK}) have expressions for all the coefficients of $P_N(X,Y)$ (obtained using non skein-theoretic techniques). Still, if one is only interested in $M_{N-1,N}$, the skein-theoretic proof above seems more direct then their proof.
\end{remark}

\subsection{The action on the kernel space}

In this part, we set $r=2N$ and let $A_r$ be a $4r$-th primitive root of unity.  
Now for $0 \leq k \leq N-1$, we can define the following vectors in $S_{A_r}(B^3,(N)_4)$ : $$ v_k = \quad \begin{minipage}[c]{2.2cm}
\includegraphics[scale = 0.15]{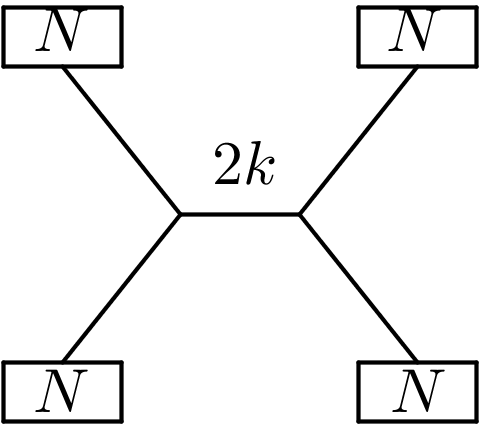}
\end{minipage}\quad \text{and} \quad v_k^{*} = \quad  \begin{minipage}[c]{2cm}
\includegraphics[scale = 0.15]{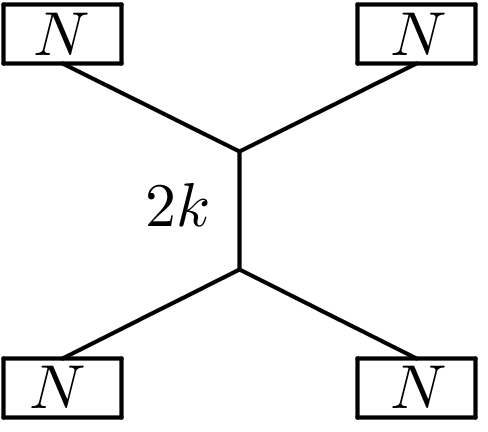}
\end{minipage}$$
These vectors are well defined since $A_r$ is a $8N$-th primitive root of unity so the $(2N-2)$-th Jones-Wenzl idempotent exists. Moreover since  the triple $(2N-2,N,N)$ is not $2r$-admissible : $v = v_{N-1}$ and $v^{*} = v^{*}_{N-1}$ both belong to $ \mathrm{K}_{r}(B^3,(N)_4)$. To prove that these two vectors are linearly independent, we have to verify some properties of Jones-Wenzl idempotents.
\\

For m an integer we denote by $TL_m$ the $m$ strands Temperley-Lieb algebra over the field of rational functions $\mathbb{C}(A)$. For $k$ an integer, we denote by $[k] = \frac{A^{2k}-A^{-2k}}{A^2-A^{-2}}$. We also define the following element in $TL_{2m}$ :
$$ t_m = \quad \begin{minipage}[c]{2cm}
\includegraphics[scale = 0.11]{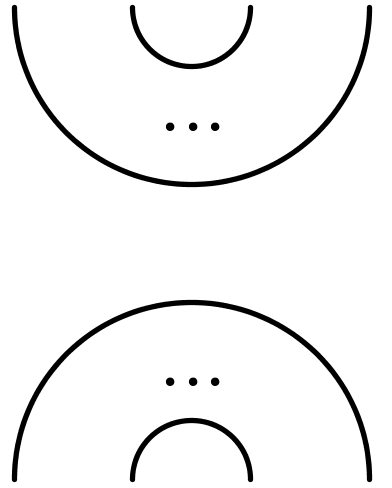}
\end{minipage}$$ \\We denote by $f^{(m)}$, the $m$-th Jones-Wenzl idempotent. For $x \in TL_{2m}$, we define $\phi_m(x)$ to be the coefficient of $t_m$ in $x$ with respect to the standard basis given by Temperley-Lieb diagrams without crossing and without trivial circles.
\begin{proposition} \label{prop5_c3} $$ \phi_m(f^{(2m)}) = \prod_{k=1}^{m} \frac{[k]^2}{[2k] [2k-1]}$$
\end{proposition}

\begin{proof}  If $m=1$, this formula is given by a straightforward computation. Let $m \geq 2$, we now use the recursive formula proved by Frenkel and Khovanov (see \cite[Theor~3.5]{FK}) and independently proved by Morrison (see \cite[Prop~3.3]{Morrison}). This formula says that for all $L \geq 2$ :

\begin{equation} \label{eq8_c3}
  \begin{minipage}[c]{2cm}
\includegraphics[scale = 0.18]{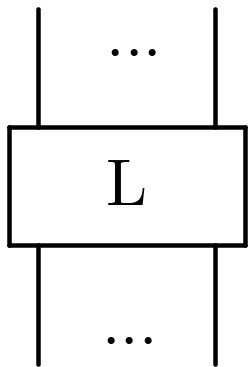}
\end{minipage}  =  \quad \begin{minipage}[c]{1.5cm}
\includegraphics[scale = 0.17]{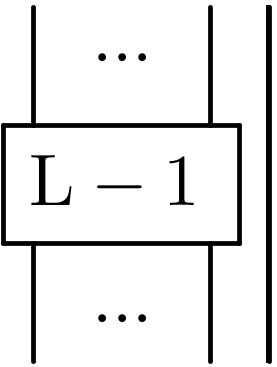}  \end{minipage}+ \,  \sum\limits_{j=1}^{L-1}  \frac{[j]}{[L]} \quad \begin{minipage}[c]{2.7cm}
\includegraphics[scale = 0.18]{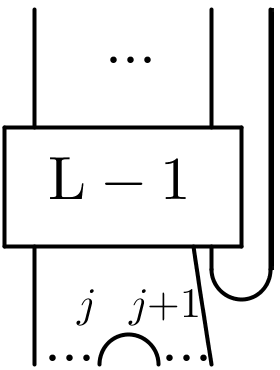}
\end{minipage} 
\end{equation}
Applying this formula twice gives 
$ \phi_{m} \left[ \quad  \begin{minipage}[c]{1.5cm}
\includegraphics[scale = 0.18]{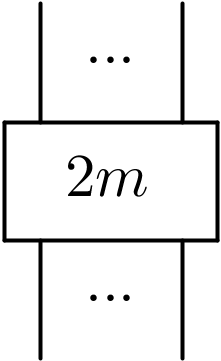}
\end{minipage}  \right]= \frac{[m]^2}{[2m][2m-1]} \quad \phi_{m} \left[ \, \,  \begin{minipage}[c]{1.6cm}
\includegraphics[scale = 0.19]{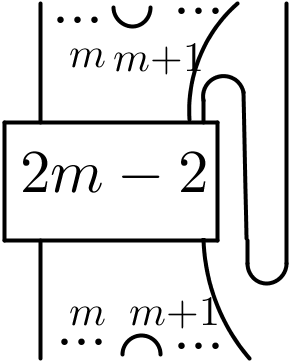} 
\end{minipage} \right]$.
The last equality means : $\phi_{m}(f^{(2m)})=\frac{[m]^2}{[2m][2m-1]} \phi_{m-1}(f^{(2m-2)})$ which proves the formula. 
\end{proof}

\begin{corollary}  \label{cor4_c3}$v$ and $v^{*}$ are two linearly independent vectors in $\mathrm{K}_r(B^3,(N)_4)$.
\end{corollary}

\begin{proof}  Let $a_k$ be defined by $v = \sum\limits_{k=0}^{N} a_k Y^k X^{N-k} $, since $v^* = s(v)$, we have $v^{*} = \sum\limits_{k=0}^{N} a_{N-k} Y^k X^{N-k} $. We remark that $a_N = 0$, so to prove the statement it is enough to prove that $a_0 \neq 0$. Using Proposition \ref{prop5_c3} we get :
$$a_0 = \phi_{N-1} (f^{(2N-2)})_{\mid_{A=A_r}} \neq 0 $$  since $[m]_{\mid_{A = A_r}} \neq 0$ for $ 1 \leq m \leq N-1$ (recall that $A_r$ is a $8N$-th primitive root of unity).
\end{proof}
We still denote by $\rho$ the evaluation of the representation $\rho : \mathrm{M}(0,4) \rightarrow \text{PAut}(S(B^3,(N)_4))$  at $A= A_r$.

\begin{proposition} \label{prop6_c3}  We denote by $\rho_{\infty}$ the restriction of $\rho$  to the space $\mathrm{K}_{r}(B^3,(N)_4)$. In the basis $(v,v^{*})$ we have :
$$ \rho_{\infty}(\sigma_1) = 
\chi_0 \, \begin{pmatrix}
   1 & \alpha_N  \\
   0 & 1
\end{pmatrix} \quad \text{and} \quad \rho_{\infty}(\sigma_2) = \chi_0 \, \begin{pmatrix}
   1 & 0 \\
   \alpha_N  & 1
\end{pmatrix}
$$where $\chi_0 =(-1)^{N-1} A^{2N(N-1)} \in \{-1,1 , -i,i \}$ and $\alpha_N \in \{ -i , i\}$.
\end{proposition}

\begin{proof} 
We know, from the previous subsection (see Remark \ref{rem}), that $\sigma_1$ acts on $S_{A_r}(B^3,(N)_4)$ with eigenvalues $c_j = (-1)^j A_r^{j(2j+2)}$ for $j = 0, ... , N$. Moreover $\sigma_1$ acts on $\mathrm{V}_{2r}(S^2,(N)_4)$ with eigenvalues\footnote{the corresponding eigenvectors are $v_0,...,v_{N-2}$.}  $c_j = (-1)^j A_r^{j(2j+2)}$ for $j = 0, ... , N-2$. Hence from $$ \mathrm{V}_{2r}(S^2,(N)_4) = S_{A_r}(B^3,(N)_4) \big{/} \mathrm{K}_r(B^3,(N)_4)$$ we deduce that $\rho_{\infty}( \sigma_1)$ has eigenvalues $c_{N-1}$ and $c_N$. Since $A_r^{4N} = -1$  we have $c_{N-1}= c_N$. Hence $\rho_{\infty}( \sigma_1)$ has a unique eigenvalue which we denote by $\chi_0 =  (-1)^{N-1} A_r^{2N(N-1)} $. Since $v$ is an eigenvector for $\rho_{\infty}(\sigma_1)$ and since $s \circ \rho_{\infty}(\sigma_1) \circ s = \rho_{\infty}(\sigma_2)$, one has in the basis $(v,v^*)$ 
$$\rho_{\infty}(\sigma_1) = 
\chi_0 \, \begin{pmatrix}
   1 & \alpha_N  \\
   0 & 1
\end{pmatrix} \quad \text{and} \quad \rho_{\infty}(\sigma_2) = 
\chi_0 \, \begin{pmatrix}
   1 & 0  \\
   \alpha_N & 1
\end{pmatrix}$$
where $\alpha_N \in \mathbb{C}$. Applying $\rho_{\infty}$ to the braiding relation $ \sigma_1 \sigma_2 \sigma_1 = \sigma_2 \sigma_1 \sigma_2$ (we can apply $\rho_{\infty}$ to this relation by the Remark \ref{R}) gives 
by a straightforward computation $$ \alpha_N=0 \quad \text{or} \quad \alpha_N^{2} = -1$$
So it remains to prove that $\alpha_N \neq 0$. Recall that $(M_{j,k})$ is the matrix of $\rho(\sigma_1)$ in the basis $(Y^k X^{N-k})_{k}$. Since $v^*$ has a non zero last coefficient in the basis $(Y^k X^{N-k})_{k}$ and since $(M_{j,k})$ is upper triangular, an elementary check shows that $\alpha_N=0$ implies $M_{N-1,N}=0$. If we use  (\ref{eq7_c3}) we deduce that  $M_{N-1,N}$ is the coefficient of $Y^{N-1}X$ in $P_N(X,Y)$ multiplied by $(-A_r)^{N^2+2N}$, therefore :
 $$M_{N-1,N}= 2 (-1)^N \left( \frac{  A_r^{2N^2+2N}}{A_r^2-A_r^{-2}} \right) \neq 0$$

\end{proof}

\begin{proof}[\textbf{Proof of Theorem \ref{main_theorem_punctured_spheres} for the four-punctured sphere}] As in Section $4$ the result follows by comparing Conclusion 1 and Proposition \ref{prop6_c3} 
\end{proof}
\begin{proof}[\textbf{Proof of Corollary \ref{cor1_c3} }]

First McMullen's representation $\rho_{-1}$ on $\mathrm{H}^1(\mathbb{T}^2)_{-1}$ is isomorphic to $ \rho_{\text{hom}} : \mathrm{M}(0,4) \rightarrow \mathrm{PSL}_2({\mathbb{Z}})$ defined in \cite{AMU} by 
$$ \rho_{\text{hom}}(\sigma_1) = \rho_{\text{hom}}(\sigma_3 )= \begin{pmatrix}
   1 & 1  \\
   0 & 1
\end{pmatrix} \quad \text{and} \quad \rho_{\text{hom}}(\sigma_2 )= \begin{pmatrix}
   1 & 0  \\
   -1 & 1
\end{pmatrix}$$ This can be seen by conjugating the matrices of Proposition \ref{prop6_c3} by the change of coordinates matrix from $(v,v^*)$ to $ (v,\alpha_N v^*)$.

We have the following from \cite[Lemma~3.6]{AMU} : $\phi \in \mathrm{M}(0,4)$ is pseudo-Anosov if and only if $\mid \text{tr}(\rho_{\text{hom}}(\phi)) \mid >2$.

 Now let $\phi \in \mathrm{M}(0,4)$ be pseudo-Anosov. Let $A_{\infty}$ be a $8N$-th primitive root of unity. Recall that $\rho$ is the representation of $\mathrm{M}(0,4)$ on $S_{A_{\infty}}(B^3,(N)_4)$. Let $Q$ be the matrix of $\rho(\phi)$ in the basis $(Y^k X^{N-k})_k$ and $Q_0$ be the matrix of $\rho_{\infty}(\phi)$ in the basis $(v, \alpha_N v^*)$. We remark that up to scale factor (which is a root of unity), the matrix $Q_0$ is the same as $\rho_{\text{hom}}(\phi)$. Since $\phi$ is pseudo-Anosov, $\rho_{\text{hom}}(\phi)$ has an eigenvalue $\lambda$ with $\mid \lambda \mid > 1$, so the same property holds for $Q_0$ and $Q$. Now take a sequence of primitive $4r$-th root of unity $A_r$ with $\lim\limits_{r \to \infty} A_r = A_{\infty}$. For $r$ big enough, $\dim(\mathrm{V}_{2r}(S^2,(N)_4)) = N+1$. Let $Q_r$ be the matrix of $\phi$ when acting on $\mathrm{V}_{2r}(S^2,(N)_4)$ in the basis $(Y^k X^{N-k})_k$. We have :
$$\lim\limits_{r \to \infty} Q_r = Q$$
Hence for $r$ big enough, the matrix $Q_r$ has an eigenvalue $\lambda_r$ with $\mid \lambda_r \mid >1$.
\end{proof}

\begin{remark} The Corollary \ref{cor1.2_c3} can be proved using the exact same method.
\end{remark}

\section{Comparison with the result of \cite{JensSoren}}

\subsection{Review of Egsgaard and Jorgensen results}
Let first recall some known fact about the mapping class group of punctured spheres, we refer to \cite{Birman} for more details. Let $\Sigma_g$ be the genus $g$ surface and let $M(g,0)$ be the mapping class group of $\Sigma_g$. Let $\iota \in M(g,0)$ be the canonical hyperelleptic involution. Let $G_{\iota}$ be the subgroup of  elements of $M(g,0)$ which commute with $\iota$. It is known that \begin{equation} \label{birman}
G_{\iota} /  \langle  \iota \rangle \quad  \simeq \quad \mathrm{M}(0, 2g+2) \end{equation} 
We remark that the isomorphism (\ref{birman}) respects the Nielsen-Thurston classification, in particular an element $\sigma \in \mathrm{M}(0,2g+2)$ is pseudo-Anosov if and only if a corresponding element $ \tilde{\sigma} \in G_{\iota} \subset M(g,0)$ (which is defined only up to a power of $\iota$) is pseudo-Anosov.
\begin{definition} \label{pA} We say that a pseudo-Anosov element $\phi \in \mathrm{M}(0,2g+2)$ is homological if a corresponding pseudo-Anosov element $\tilde{\phi} \in \mathrm{M}(g,0)$ has an orientable measured foliation. 
\end{definition}
  The following result is from \cite[Theorem~1.4]{JensSoren} and is the analog of Corollary \ref{cor1.2_c3} in our setting.

\begin{theorem}[Egsgaard, Jorgensen] \label{theorem_JensSoren}  If $\phi \in \mathrm{M}(0,n)$ is a homological pseudo-Anosov element then there exists $r_0$ such that for $r \geq r_0$, $\phi$ has infinite order when acting on $\mathrm{V}_{2r}(S^2 ,(1)_{n})$.
\end{theorem}


\subsection{Comparison of Corollary  \ref{cor1.2_c3} with Theorem \ref{theorem_JensSoren}}

In this short subsection we give experimental examples in which Corollary \ref{cor1.2_c3} can be applied. We also show the differences between the criteria coming from Corollary \ref{cor1.2_c3} and Theorem \ref{theorem_JensSoren}. Note that Theorem \ref{theorem_JensSoren} only covers the case $N=1$ whereas Corollary \ref{cor1.2_c3} holds for all $N \geq 1$.

Note that if $n \geq 6$ even, $N \geq  1$ and $A$ is a $2Nn$ primitive root of unity, one can check that the generators $\sigma_1,...,\sigma_{n-1} \in \mathrm{M}(0,n)$ all have order $n$ when acting on $\mathrm{H}^{1}(X)_{A^{-4N}}$ (note this is not true when $n=4$). The element $\phi = \sigma_1^{n} \sigma_2^{-n} ... \sigma_{n-1}^{n} \in \mathrm{M}(0,n)$ is a homological pseudo-Anosov element (see \cite{Penner}) but acts trivially on $\mathrm{H}^{1}(X)_{A^{-4N}}$, hence Theorem \ref{theorem_JensSoren} applies to $\phi$ but Corollary \ref{cor1.2_c3} cannot be used for $\phi$. 

\begin{definition} \label{torelli} We say that an element $\sigma \in \mathrm{M}(0,2g+2)$ is in the Torelli group if a corresponding element $\tilde{\sigma} \in G_{\iota} \subset M(g,0)$ has a trivial action on $\mathrm{H}_1(\Sigma_g,\mathbb{Z})$.
\end{definition}
The results proved in \cite{JensSoren} cannot apply to pseudo-Anosov elements in the Torelli group (according to Definition \ref{torelli}). We give a family $(\phi_{k,l})$ of pseudo-Anosov in the Torelli group where Corollary \ref{cor1.2_c3} applies to prove that the AMU conjecture holds for all $(\phi_{k,l})$. One interesting feature of Corollary \ref{cor1.2_c3} is that certain Dehn twists in the Torelli group (according to Definition \ref{torelli}) act with infinite order and are quasi-unipotent (a certain power is a unipotent matrix). For instance, for $n=6$  the full Dehn twist $\delta = (\sigma_1  \sigma_2)^3$ is unipotent and is in the Torelli group. Moreover according to Thurston's construction (see \cite[Theorem~7]{Thurston}), for any $k, l \neq 0$ integers, the element $$\phi_{k,l} = \delta^k (s_3 s_6) \delta^l (s_3 s_6)^{-1}$$ is pseudo-Anosov and is clearly in the Torelli group (since $\delta$ is in the Torelli group). A straightforward computation shows that the trace of $\phi_{k,l}$ is $-12 k l +4$ when acting on $\mathrm{H}^{1}(X)_{q}$ for $q = \exp(\frac{-i \pi}{3})$. Now we see that for any non zero integers $k$ and $l$ $$|-12 k l +4 | > 4 = \dim \mathrm{H}^{1}(X)_{e^{-i \pi/3}}$$ so $\phi_{k,l}$ has a spectral radius strictly greater than one when acting on $\mathrm{H}^{1}(X)_{q}$ for $q = \exp(\frac{-i \pi}{3})$.

Institut Math\'ematiques de Jussieu (UMR 7586 du CNRS). 

\'Equipe de Topologie et G\'eom\'etrie Alg\'ebrique,

 Case 247, 4 pl.Jussieu, 75252 Paris Cedex 5, France.

\emph{Email address}: \textbf{ramanujan.santharoubane@imj-prg.fr}
\end{document}